\RequirePackage[l2tabu, orthodox]{nag}
\documentclass[10pt, draft]{article}

\usepackage{amsmath,amssymb,amsbsy,amsfonts,amsthm,latexsym,amsopn,amstext,
                            amsxtra,euscript,amscd}
\usepackage{booktabs}
\usepackage{enumitem}
\usepackage{tikz}
\usepackage[norefs,nocites]{refcheck}
\numberwithin{equation}{section}
\newtheorem{theorem}{Theorem}[section]

\newtheorem{thm}[theorem]{Theorem}

\newtheorem{lem}[theorem]{Lemma}

\def\\{\cr}
\def\({\left(}
\def\){\right)}
\def\[{\left[}
\def\]{\right]}
\def\<{\langle}
\def\>{\rangle}
\def\fl#1{\left\lfloor#1\right\rfloor}

\def\N{\mathbb{N}}

\def\R{\mathbb{R}}

\def\notdivides{\mathrel{\kern-3pt\not\!\kern3.5pt\bigm|}}

\begin{document}

%\title{\textbf{}}

\title{\Large \textbf{Cardinality of a floor function set}}
\author{%
%\scshape{JUAN ARIAS DE REYNA}\\
%{Department of Mathematical Analysis, Seville University} \\
%{Seville, Spain}\\
%\texttt{ arias@us.es}
%\and
\scshape {RANDELL HEYMAN}  \\
School of Mathematics and Statistics,\\ University of New South Wales \\
Sydney, Australia\\
\texttt {randell@unsw.edu.au}
}

%\date{ }
\maketitle
%\newpage
%\begin{center}
%Received (20 March 2014)\\
%Accepted (             )
%\end{center}

%\tableofcontents

\begin{abstract}
Fix a positive integer X. We quantify the cardinality of the set $\{\fl{X/n}\}_{n=1}^X$. We discuss restricting the set to those elements that are prime, semiprime or similar.
\end{abstract}

%Keywords: Divisors, greatest common divisor
%\newline

%AMS Classification: XXX
%\section{Crypto}
%We wish to count the number of integers $n$ with
%\begin{itemize}
%\item $n \le 2^{4096}$
%\item $n=n_sn_l$
%\item $P^-(n_s)>2^{200}$
%\item $P^+(n_s) \le 2^{640}$
%\item $P^-(n_l)>2^{640}$
%\item $n_l$ is composite
%\item $n_l>2^{2048}$
%\end{itemize}

\section{Introduction}
Throughout we will restrict the variables $m$ and $n$ to positive integer values. For any real number $X$ we denote by $\fl{X}$
its integer part, that is, the greatest integer that does not exceed $X$.
The most straightforward sum of the floor function is related to the divisor summatory function since
$$\sum_{n \leqslant X} \fl{\frac{X}{n}}=\sum_{n \leqslant X}\sum_{k \leqslant X/n}1=\sum_{n \leqslant X}\tau(n),$$
where $\tau(n)$ is the number of divisors of $n$. From \cite[Theorem~2]{BouWat} we infer
$$\sum_{n \leqslant X} \fl{\frac{X}{n}}= X \log X + X (2 \gamma - 1) + O\( X^{517/1648+ o(1)}\),$$
where $\gamma$ is  the \textit{Euler--Mascheroni} constant, in particular $\gamma  \approx 0.57\, 722$.

Recent results have generalised this sum to
$$\sum_{n \leqslant X} f\(\fl{\frac{X}{n}}\),$$
where $f$ is an arithmetic function (see \cite{Bor}, \cite{Chern} and \cite{Gos}).

In this paper we take a different approach by examining the cardinality of the set
$$S(X):=\left\{m:m=\fl{\frac{X}{n}} \textrm{for some } n \le X\right\}.$$

Or main results are as follows.
\begin{thm}
\label{thm:set cardinality}
Let $X$ be a positive integer and let
$$b=\frac{-1+\sqrt{4X+1}}{2}.$$
We have
$$|S(X)|=\fl{b}+\fl{\frac{X}{\fl{b+1}}}.$$
\end{thm}

\begin{thm}
\label{thm:set cardinality 2}
We have
$$|S(X)|=2\sqrt{X}+O(1).$$
\end{thm}

\section{Proof of Main Theorems}
\label{sec:main theorem}
Throughout let
\begin{align}
\label{eq:b}
b=\frac{-1+\sqrt{4X+1}}{2},
\end{align}
and note that
$$\frac{X}{b}=b+1.$$
We define 2 sets:
\begin{align}
\label{eq:m less than equal b}
S_1(X)=\left\{m:m=\fl{\frac{X}{n}}, m\le b \right\}
\end {align}
and
\begin{align}
\label{eq:S greater than b}
S_2(X)=\left\{m:m=\fl{\frac{X}{n}}, n(n-1)\le X \right\}.
\end{align}
%We note that if $n(n-1)<X$ then
%$$n< \frac{\sqrt{4X+1}}{2}=\frac{X}{b}$$
%which implies that
%\begin{align}
%\label{eq:from S2}
%\frac{X}{n}&> b.
%\end{align}
We will quantify $S_1(X)$ and then show that $S_1(X) \cup S_2(X) \subseteq S(X)$. This will allow us to use the inclusion-exclusion principle once we quantify $|S_2(X)|$ and $|S_1(X)\cap S_2(X)|$.

We start by calculating the number of elements of $S_1(X)$. Let $m$ be an arbitrary positive integer with
$$m \le b=\frac{-1+\sqrt{4X+1}}{2}.$$
%\begin{align}
%\label{eq:S-X}
%m&=\fl{\frac{X}{n}}\le\frac{X}{n}\le \frac{X}{a}=b=\frac{-1+\sqrt{4X+1}}{2}.
%\end{align}
This means that $$m^2+m-X\le 0$$ which implies that $m(m+1) \le X$. Thus
$$\frac{X}{m(m+1)}\ge 1$$
and therefore
$$\frac{X}{m}-\frac{X}{m+1}\ge 1.$$
Since the interval from $\frac{X}{m}$ to $\frac{X}{m+1}$ is at least 1 there must be an integer $n$ such that
\begin{align}
\label{eq:S-X2}
\frac{X}{m+1}<n\le \frac{X}{m},
\end{align}
from which
$$X-n < mn\le X.$$
In turn this implies that
$$\frac{X}{n}-1<m \le \frac{X}{n}.$$
This means that $m=\fl{X/n}$ and so $m \in S_1(X)$. From \eqref{eq:m less than equal b} there are $\fl{b}$ possible values of $m$. From \eqref{eq:S-X2} we see that can always find an $n$ to give us any of these values of $m$.
Therefore the numbers $1,2,\ldots,\fl{b}$ are the only elements of $S_1(X)$ and so
\begin{align}
\label{eq:first part}
|S_1(X)|= \fl{b}.
\end{align}
Apart from quantifying $S_1(X)$ we also note that the fact that $1,2,\ldots,\fl{b}\in S_1(X)$ implies that $S_1(X) \subseteq S(X)$. By reference to the definitions of $S_2(X)$ and $S(X)$ we see that $S_2(X) \subseteq S(X)$. Thus
$$S_1(X) \cup S_2(X) \subseteq S(X)$$ and so, using the inclusion-exclusion principle,
\begin{align}
\label{eq:SX}
|S(X)|=|S_1(X)|+|S_2(X)|-|S_1(X)\cap S_2(X)|.
\end{align}

We now consider the cardinality of $S_2(X)$.
We show that $n(n-1)<X$ implies $\fl{X/n}$ and $\fl{X/(n-1)}$ are distinct. We have
$$\fl{\frac{X}{n-1}}-\fl{\frac{X}{n}}=\frac{X}{n-1}-\frac{X}{n}-\left\{\frac{X}{n-1}\right\}+\left\{\frac{X}{n}\right\},$$
where $\{\cdot\}$ represents, as usual, the fractional part of the real number.
So
\begin{align}
\label{eq:high m values}
\fl{\frac{X}{n-1}}-\fl{\frac{X}{n}}&=\frac{X}{n(n-1)}+t,
\end{align}
where $t\in (-1,1)$.
Recalling that $n(n-1)\le X$ we have
$$\frac{X}{n(n-1)}\ge 1.$$
Substituting into \eqref{eq:high m values} we see that
$$\fl{\frac{X}{n-1}}-\fl{\frac{X}{n}}>0$$
which implies that $\fl{X/n}$
and $\fl{X/(n-1)}$ are distinct.
Since $n(n-1)\le X$ we have, solving the quadratic equation,
\begin{align}
\label{eq:Xb}
n& \le \frac{X}{b}
\end{align}
and so
\begin{align}
\label{eq:second part}
|S_2(X)|=\fl{\frac{X}{b}}=\fl{b}+1.
\end{align}

To finish the proof it only remains to consider $|S_1(X)\cap S_2(X)|$.
We have seen that
$$S_1(X)=\{1,2,\cdots,\fl{b}\}.$$
From \eqref{eq:Xb} we see that
$$n \le \frac{X}{b}=b+1.$$
So the values of $n$ in $S_2(X)$ are $1,2,\cdots\fl{b+1}$ and therefore
$$S_2(X)=\Bigg\{\fl{\frac{X}{\fl{b+1}}},\fl{\frac{X}{\fl{b}}},\cdots,X\Bigg\}.$$
The set $S_1(X) \cap S_2(X)$ will be non empty if
$$\fl{\frac{X}{\fl{b+1-c}}}=\fl{b-d},$$
for some $c,d\ge 0.$
From this we deduce that
$$\frac{X}{b+1-c}<b+1-d,$$
and so
$$X<((b+1)-d)(b+1-c).$$
Recalling that $X=b(b+1)$ we have
$$b+1-c(b+d+1)-d(b+1)>0,$$
which is only possible if $c=d=0$.
%the smallest element of $S_2(X)$ is
%$$\fl{\frac{X}{\fl{b+1}}}.$$
Thus there will be at most one element of $S_1(X) \cap S_2(X)$ and this one element will occur if, and only if,
$$\fl{\frac{X}{\fl{b+1}}}=\fl{b}.$$
In fact,
$$|S_1(X)\cap S_2(X)|=1-\(\fl{\frac{X}{\fl{b+1}}}-\fl{b}\).$$

%Suppose $c \in S_1(X)\cap S_2(X)$. Then, from \eqref{eq:Xb} we have
%$$\frac{X}{n}>b,$$
%from which
%\begin{align}
%\label{eq:m1}
%\frac{X}{n}> \sqrt{X+\frac{1}{4}}-\frac{1}{2}.
%\end{align}
%On the other hand from $S_1(X)$ we have that
%\begin{align}
%\label{eq:m2}
%\fl{\frac{X}{n}}=m\le b= \sqrt{X+\frac{1}{4}}-\frac{1}{2}.
%\end{align}
%It is clear that there will be at most one value of $m$ that satisfies $\eqref{eq:m1}$ and $\eqref{eq:m2}.$ Specifically,
%\begin{align*}
%|S_1(X)\cap S_2(X)|= \begin{cases} 1& \text{if there exists a } n \text{ such that } \frac{X}{n}>b \text{ and }\fl{\frac{X}{n}}\le b
%, \\
%0 & \text{otherwise}.
%\end{cases}
%\end{align*}
Combining this equation with $\eqref{eq:SX}, \eqref{eq:first part}$ and $\eqref{eq:second part}$  and simplifying completes the proof of Theorem \ref{thm:set cardinality}.

Theorem \ref{thm:set cardinality 2}, that is
$$|S(X)|=2\sqrt{X}+O(1),$$
follows immediately from Theorem \ref{thm:set cardinality}.

\section{Discussion}
We can generalise $S(X)$ by considering elements of $S(X)$ that are divisible by some positive integer $d\le X$.
This is interesting in its own right but could also form the basis for calculating something much more interesting; the number of primes, semi primes or similar in $S(X)$.

Let $$S_d(X)=\left\{m:m=\fl{\frac{X}{n}} \textrm{for some } n \le X,d\mid\fl{\frac{X}{n}} \right\}.$$

A standard approach to express $S_d(X)$ would be to follow a path involving an indicator function, differences of floor functions, the $\psi$ function and exponential sums, hoping that we can bound the exponential sums (here $\psi(y)=y-\fl{y}-1/2$). Unfortunately this is not the case here. The process yields
\begin{lem}
\label{lem:number of divisors}
\begin{align*}
|S_d(X)|&=\frac{4X^{1/2}}{3d}+\sum_{r=1}^{\fl{\frac{X}{b}}}\sum_{j=\frac{X-r}{rd}}^\frac{X}{rd}\( \psi\(\frac{X}{dj+1}\)-\psi\(\frac{X}{dj}\)\)+O(1),
\end{align*}
where $a=X/b$.
\end{lem}
A proof is given in Section \ref{sec:no of divisors}.
Calculating various sums using Maple suggests that the double sum cannot successfully be bound. In fact Maple suggests that the double sum is asymptotically equivalent to $2X^{1/2}/3d$. If this argument is correct then
$$S_d(X)\sim\frac{2X^{1/2}}{d},$$
as one would expect heuristically.

\section{Trivial bounds}
In the absence of a better approach we outline some trivial bounds on $|S_d(X)|$. The interested reader may wish to improve these bounds.
\begin{thm}
For a real positive $X$ and a positive integer $d\le X$ with $d\ne 1$ we have
$$\frac{X^{1/2}}{d}+O(1)\le |S_d(X)|\le \frac{X}{2}+O(1).$$
%If $d|X$ then
%$$\max\Bigg\{\tau\(\frac{X}{d}\),X^{1/2}+O(1)\Bigg\}\le |S_d(X)|\le \frac{X}{2}+O(1).$$
\end{thm}
\begin{proof}
The lower bound follows from the fact that $1,2,\ldots,\fl{b}\in S(X)$ (see Section \ref{sec:main theorem}). Of these $\fl{\fl{b}/{d}}$ will be divisible by $d$. Recalling that $b=X^{1/2}+O(1)$ the result follows. The upper bound flows from the fact that of the $X$ numbers in the sequence $(\fl{X/n})_{n=1}^X$ the number 1 appears $X/2$ times if $X$ is even and $(X+1)/2$ times if $X$ is odd.
%For the second lower bound we deduce that if $X$ is divisible by $d$ then every divisor of $X/d$ will be an element of $S_d(X)$.
\end{proof}
\section{Proof of Lemma \ref{lem:number of divisors}}
To simplify notation we will let
$$a=\frac{X}{b}.$$
\label{sec:no of divisors}
It is clear that
\begin{align}
\label{eq:SdX}
S_d(X)=S_d^-(X) \cup S_d^+(X),
\end{align}
where
$$S_d^-(X)=\left\{m:m=\fl{\frac{X}{n}} \textrm{for some } n <a,d|\fl{\frac{X}{n}} \right\}$$
and
$$S_d^+(X)=\left\{m:m=\fl{\frac{X}{n}} \textrm{for some } n \ge a,d|\fl{\frac{X}{n}}\right\}.$$
From Section \ref{sec:main theorem} it is clear that the numbers $1,2,\ldots,\fl{b}$ will be elements of  $S^+(x)$. Of these exactly $\fl{\fl{b}/d}$ will be divisible by $d$ and so
\begin{align}
\label{eq:Sdplus}
|S_d^+(x)|=X^{1/2}+O(1).
\end{align}
We now quantify $S_d^-(X)$. We observe that if
$$\frac{X}{dj+1}\le n < \frac{X}{dj}$$ then
$$dj< \frac{X}{n} \le dj+1$$ and so
$$\fl{\frac{X}{n}}=dj \text{ for some }j,$$
which implies that
$$m=\fl{\frac{X}{n}} \in S_d^-(X).$$
Furthermore,
$$|S_d^-(X)|=\sum_{\substack{\frac{X}{dj+1}< n \le \frac{X}{dj}\\n<a}}1,$$
if the elements of $\{\fl{X/n}\}_1^{\fl{a}}$ are distinct. To see that this condition is true we note that $n<a$ implies that $n(n+1)<X$ which means that $X/n(n+1)>1$. Thus $\fl{X/n}-\fl{X/(n+1)}>0$ which proves distinctiveness.

Next, since $n<a$ we also have that $$\frac{X}{dj}-\frac{X}{dj+1}=\frac{X}{m(m+1)}<\frac{X}{m^2}<1,$$
the last inequality being justified by the fact that $n<a$ implies that $m>X^{1/2}$. This means that there can only be one value of $n$ between
$$\frac{X}{dj} \text{ and }\frac{X}{dj+1}.$$
Therefore
$$|S_d^-(X)|=\sum_{r=1}^{\fl{a}}\sum_{j=\frac{X-r}{rd}}^{\frac{X}{rd}} \textbf{1}(j),$$
where
$$
\textbf{1}(j)=
\begin{cases}
1 &\text{if }X/(dj+1)< n \le X/dj \text{ for some } n \in \N, \\
0 & {\text{otherwise}}.
\end{cases}
$$
We can replace the indicator function with floor functions as follows:
\begin{align}
\label{eq:indicator}
|S_d^-(X)|&=\sum_{r=1}^{\fl{a}}\sum_{j=\frac{X-r}{rd}}^{\frac{X}{rd}}\fl{\frac{X}{jd}}-\fl{\frac{X}{dj+1}}.
\end{align}

For any real $t \in \R$ we denote
$$\psi(t)=t-\fl{t}-\frac{1}{2}.$$
Replacing the floor functions in \eqref{eq:indicator} with the $\psi$ function we obtain
\begin{align}
\label{eq:Sdminus}
|S_d^-(X)|&=\sum_{r=1}^{\fl{a}}\sum_{j=\frac{X-r}{rd}}^\frac{X}{rd}\( \frac{X}{dj}-\frac{X}{dj+1}+\psi\(\frac{X}{dj+1}\)-\psi\(\frac{X}{dj}\)\)
\\&=S_1+S_2\notag,
\end{align}
where
$$S_1=\sum_{r=1}^{\fl{a}}\sum_{j=\frac{X-r}{rd}}^\frac{X}{rd}\( \frac{X}{dj}-\frac{X}{dj+1}\)$$
and
$$S_2=\sum_{r=1}^{\fl{a}}\sum_{j=\frac{X-r}{rd}}^\frac{X}{rd}\( \psi\(\frac{X}{dj+1}\)-\psi\(\frac{X}{dj}\)\).$$
Estimating $S_1$ we have
\begin{align}
\label{eq:S1}
S_1&=\sum_{r=1}^{\fl{a}}\sum_{j=\frac{X-r}{rd}}^\frac{X}{rd}\frac{X}{dj(dj+1)}\notag\\
&=\frac{X}{d}\sum_{r=1}^{\fl{a}}\sum_{j=\frac{X-r}{rd}}^\frac{X}{rd}\frac{1}{dj^2}-\frac{X}{d}\sum_{r=1}^{\fl{a}}\sum_{j=\frac{X-r}{rd}}^\frac{X}{rd} \frac{1}{dj^2(dj+1)}\notag\\
&=\frac{X}{d^2}\sum_{r=1}^{\fl{a}}\sum_{j=\frac{X-r}{rd}}^\frac{X}{rd}\frac{1}{j^2}+O\(X\sum_{r=1}^{\fl{a}}\sum_{j=\frac{X-r}{rd}}^\frac{X}{rd}\frac{1}{j^3}\).
\end{align}
We now estimate
$$\frac{X}{d^2}\sum_{r=1}^{\fl{a}}\sum_{j=\frac{X-r}{rd}}^\frac{X}{rd}\frac{1}{j^2}.$$

Using Abel summation we have
\begin{align*}
\sum_{j=\frac{X-r}{rd}}^\frac{X}{rd}\frac{1}{j^2}&=\frac{(X-r)/rd}{((X-r)/rd)^2}-\frac{X/rd}{(X/rd)^2}-\int_{(X-r)/rd}^{X/rd}t\(\frac{-2}{t^3}\) dt\\
&=\frac{1}{(X-r)/rd}-\frac{1}{X/rd}+2\int_{(X-r)/rd}^{X/rd}\frac{1}{t^2} dt\\
&=-\frac{r^2d}{X(X-r)}+2\frac{r^2d}{X(X-r)}\\
&=\frac{r^2d}{X(X-r)}.
\end{align*}
So
\begin{align*}
\frac{X}{d^2}\sum_{r=1}^{\fl{a}}\sum_{j=\frac{X-r}{rd}}^\frac{X}{rd}\frac{1}{j^2}&=\frac{X}{d^2}\sum_{r=1}^{\fl{a}}\frac{r^2d}{X(X-r)}\\
&=\frac{1}{d}\sum_{r=1}^{\fl{a}}\frac{r^2}{X-r}.
\end{align*}
Using Abel summation again we have
\begin{align*}
\frac{X}{d^2}\sum_{r=1}^{\fl{a}}\sum_{j=\frac{X-r}{rd}}^\frac{X}{rd}\frac{1}{j^2}&=
\frac{1}{d}\[\frac{\fl{a}(\fl{a}+1)(2\fl{a}+1)}{6(X-\fl{a})}-\int_1^{\fl{a}}\frac{u(u+1)(2u+1)}{6(X-u)^2}du\].
\end{align*}
Observe that $\fl{a}=X^{1/2}+O(1)$. Thus
\begin{align}
\label{eq:S1 1}
\frac{X}{d^2}\sum_{r=1}^{\fl{a}}\sum_{j=\frac{X-r}{rd}}^\frac{X}{rd}\frac{1}{j^2}
&=\frac{1}{6d}\[\frac{2X^{3/2}+O(X)}{X-X^{1/2}+O(1)}+O(1)\]\notag\\
&=\frac{X^{1/2}}{3d}+O(1).
\end{align}
Using a similar analysis we have
\begin{align}
\label{eq:S1 2}
O\(X\sum_{r=1}^{\fl{a}}\sum_{j=\frac{X-r}{rd}}^\frac{X}{rd}\frac{1}{j^3}\)&=O(X^{-1/2}).
\end{align}
Substituting \eqref{eq:S1 1} and \eqref{eq:S1 2} into \eqref{eq:S1} we conclude that
\begin{align}
\label{eq:S1 final}
S_1=\frac{X^{1/2}}{3d}+O(1).
\end{align}
substituting this expression for $S_1$ into \eqref{eq:Sdminus} and then \eqref{eq:Sdminus} and \eqref{eq:Sdplus} into \eqref{eq:SdX} completes the proof.
\section{Acknowledgements}
The author thanks William Banks for suggesting the problem (which follows naturally from \cite{Bor}), for explaining the process outlined in Section \ref{sec:no of divisors} and for his hospitality during a very pleasant stay in Missouri. The author also thanks Igor Shparlinski and Olivier Bordell\`{e}s for some useful comments.

\makeatletter
\renewcommand{\@biblabel}[1]{[#1]\hfill}
\makeatother

\end{document}